\documentclass[10pt]{amsart}
\usepackage{fullpage} 
\usepackage{amssymb} 
\usepackage{graphicx} 
\usepackage{amsmath,amsthm,amssymb,latexsym} 
\usepackage{enumerate} 
\usepackage{tikz}
\usepackage{accents} 
\usepackage[bookmarks,colorlinks,breaklinks]{hyperref}
\hypersetup{linkcolor=black,citecolor=black,filecolor=black,urlcolor=black}

\newtheorem{theorem}{Theorem}[section]

\theoremstyle{definition}

\theoremstyle{remark}

\numberwithin{equation}{section}

\begin{document}

\setcounter{page}{1}
\renewcommand{\thefootnote}{\fnsymbol{footnote}}
\title[]{On the existence of ratio limits of weighted {\em n}-generalized Fibonacci sequences with arbitrary initial conditions}
\author{Igor Szczyrba}
\address{School of Mathematical Sciences\\
                University \!of Northern Colorado\\
                Greeley CO 80639, U.S.A.}
\email{igor.szczyrba@unco.edu}

\begin{abstract}
We study ratio limits of the consecutive terms of weighted $n$-generalized Fibonacci sequences generated from arbitrary complex initial conditions by linear recurrences with arbitrary complex weights. We prove that if the characteristic polynomial of such a linear recurrence is asymptotically simple, then the ratio limit exists for {\em any\,} sequence generated from arbitrary nontrivial initial conditions and it is equal to the unique zero of the characteristic polynomial. 
\end{abstract}
 
\maketitle

\section{Introduction}
\label{sec:I} 
Sequences generated by linear recurrences of an arbitrary order $n\ge2$ and the ratio limits of their consecutive terms have been studied for over a half a century. The unweighted $n$-generalized Fibonacci sequences\footnote{They are also referred to as $n$-step Fibonacci sequences \cite{Noe}.} with weights $(1,\dots,1)$ and initial conditions $(0,\dots,0,1)$ were introduced and investigated in 1960 by Miles \cite{Miles}. In 1967, Fielder introduced and studied the unweighted $n$-generalized Lucas sequences that are generated from the initial conditions $(-1,\dots,-1,n)$ \cite {Fiedler}.  

Also in 1967, Byrd \cite{Byrd} and, independently, Flores \cite{Flores} showed that when $n$ goes to infinity, the ratio limits of the unweighted $n$-generalized Fibonacci sequences converge to 2. We recently extended this result \cite{Szczyrba} to the weighted $n$-generalized Fibonacci sequences with weights $(p,\dots,p)$, $p>0$, and initial conditions $(0,\dots,0,1)$ by proving that, when $n$ goes to infinity, the ratio limits of these sequences converge to $p+1$. In a subsequent paper \cite{Szczyrba1}, we further showed that for any given $p$ and arbitrary $n$, these ratio limits can be represented geometrically using dilations of a collection of convex compact sets with rising dimensions $n$, such as $n$-balls, $n$-pyramids, $n$-cones, or $n$-simplexes. 

An extensive study of ratio limits of weighted $n$-generalized Fibonacci sequences with complex weights and arbitrary complex initial conditions was conducted in 1997 by Dubeau et al.~\cite {Dubeau1}. They proved among the other things that if the characteristic polynomial of a given linear recurrence is asymptotically simple, i.e., among the polynomial's zeros of maximal modulus there is a unique zero $\lambda_0$ of maximal multiplicity $\nu$, then the ratio limit of the sequence $(F^{\,0}_k)_{k=-n+1}^\infty$, generated by this linear recurrence from the initial conditions $(0,\dots,0,1)$, exists and is equal $\lambda_0$. The authors also showed that if a sequence $(F^{\,\bf a}_k)_{k=-n+1}^\infty$ generated from arbitrary complex initial conditions ${\bf a}=(a_{-n+1},\dots,a_0)$ by a linear recurrence with an asymptotically simple characteristic polynomial satisfies
\begin{equation}
  \lim_{k\to\infty}\big(F^{\bf a}_k/k^{\nu-1}\lambda^k_0\big)\ne0,
  \label{11}  
\end{equation} 
then the sequence's ratio limit exists and equals $\lambda_0$ as well.

In this paper, we show that condition \eqref{11} is redundant, i.e., we prove that if the characteristic polynomial of a given linear recurrence with complex weights is asymptotically simple, then the ratio limit exist and is equal to $\lambda_0$ for {\em any\,} sequence generated by this linear recurrence from arbitrary complex initial conditions.

\section{Main result}
\label{sec:II} 
Given a linear recurrence of an order $n\ge2$ with the characteristic polynomial   
\begin{equation}
  \lambda^{n}\!-b_1\lambda^{n-1}\!-\cdots-b_{n}
  \label{21}  
\end{equation} 
with complex weights $b_i$, $i=1,\dots,n$, such that $b_n\neq0$.
We study sequences\, $(F^{{\bf a }}_k)_{k=-n+1}^\infty$\, generated by this liner recurrence from arbitrary nontrivial complex initial conditions 
${\bf a}=(a_{-n+1},\dots,a_0)$, i.e.,
\begin{equation}
F^{\,\bf a}_k= b_1F^{\bf a}_{k-1}+\cdots + b_nF^{\bf a}_{k-n},\,\,\, k>0,\,\,\, \text{and}\,\,\, F^{\bf a}_k\equiv a_k,\,\,\, k=-n+1,\dots, 0. \label{22}
\end{equation} 
If the ratio limit of the consecutive terms of a sequence  $(F^{{\bf a }}_k)_{k=-n+1}^\infty$ exists, we denote it as 
\begin{equation}\Phi^{\bf a}=\lim_{k_0<k\to\infty}\big(F^{\bf a}_{k+1}/F^{\bf a}_{k}\big), \,\,\,  \text{where}\,\,\, F^{\bf a}_k\neq0 \,\,\, \text{for} \,\,\, k > k_0 \label{23}.
\end{equation}

The assumptions that the initial conditions are nontrivial and that the weight $b_n\neq0$ imply that no more than $n-1$ consecutive elements of the sequence $(F^{{\bf a }}_k)_{k=-n+1}^\infty$ are equal to zero. Let $F^{\bf a}_{k'}$ be a nonzero element. Obviously, finding the ratio limit $\Phi^{\bf a}$ of the sequence $(F^{{\bf a }}_k)_{k=-n+1}^\infty$ is equivalent to finding the ratio limit $\Phi^{\bf a'}\!=\Phi^{\bf a}$ of the sequence $(F^{{\bf a'}}_{k})_{k=-n+1}^\infty$ with initial conditions\, ${\bf a'}=(a'_{-n+1},\dots,a'_0)=(F^{\bf a}_{k'},\dots, F^{\bf a}_{k'+n-1})$, i.e, where $a'_{-n+1}\neq0$.

\begin{theorem} 
Given a linear recurrence with complex weights $(b_{1},\dots,b_n)$, $b_n\neq0$, let $(F^{{\bf a }}_k)_{k=-n+1}^\infty$ be a sequence generated by this linear recurrence from arbitrary complex initial conditions\, ${\bf a}=(a_{-n+1},\dots,a_{-n+1})$, $a_{-n+1}\neq0$, and let $(F^{0}_k)_{k=-n+1}^\infty$ be a sequence generated from the initial conditions\, $(0,\dots,0,1)$. 

\smallskip(i) If there exists $k_0$ such that $F^{0}_k\neq0$ for $k>k_0$, then $F^{ \bf a }_k\neq0$ for $k>k_0+n-1.$

\smallskip(ii) If the characteristic polynomial of the linear recurrence is asymptotically simple, then the ratio limit\, $\Phi^{\bf a}$ exists for any nontrivial initial conditions\, ${\bf a}$ and it coincides with the characteristic polynomial's unique zero $\lambda_0$.
\end{theorem} 
\begin{proof} 
(i) 
Dubeau et al.~\cite{Dubeau1} showed that elements of the sequence $(F^{\bf a}_k)_{k=0}^\infty$ can be expressed by elements of the sequence  $(F^{ 0 }_k)_{k=-n+1}^\infty$ in the following way:
\begin{equation}
F^{\bf a}_{k}= a_0F^0_{k}+\sum_{i=1}^{n-1}a_{-i}\sum^{n-i}_{j=1}b_{i+j}F^0_{k-j}. 
\label{24}
\end{equation}
Let us assume that $F^{\bf a}_{k_1}=0$\, for some $k_1>k_0+n-1$. Due to our assumptions, the sum in the RHS of equation \eqref{24} with $k=k_1$ includes the nonzero term $a_{-n+1}b_nF^0_{k_1-n+1}$. Thus, the sum can equal zero only if it contains {\em at least one more\,} nonzero term so that all the terms add up to zero. 

Consequently, the following linear relation between the sequence elements $F^{0}_l\!$, $l=k_1-n+1,\dots,k_1$, is implied by our assumption that $F^{\bf a}_{k_1}=0$\, for $k_1> k_0+n-1$:
\begin{equation}
 F^{0}_{k_1-n+1}=c_1F^{0}_{k_1}+\cdots+c_{n-2}F^{0}_{k_1-n+2}, 
\label{25}
\end{equation}
where coefficients\, $c_i$ are determined by weights $(b_2,\dots,b_n)$ and initial conditions\, $(a_{-n+1},\dots,a_0)$, and {\em at least one} of  the $c_i$s is not equal to zero. 

Inserting the relation \eqref{25} in formula \eqref{22} with $k=k_1+1$ and initial conditions $(0,\dots,0,1)$ allows us to express the sequence element $F^{0}_{k_1+1}$ by its predecessors without using the sequence element $F^{0}_{k_1-n+1}$, i.e, relying on a new set of weights $(b'_{1},\dots,b'_{n-1},0)$. We conclude by induction that the latter is also true for all elements of the sequence $(F^{0}_k)_{k=k_1+1}^\infty$. But this contradicts our assumption that characteristic polynomial \eqref{21} includes $b_n\neq0$.

\medskip(ii) If the characteristic polynomial of the linear recurrence is asymptotically simple, the ratio limit $\Phi^{0}=\lim_{k_0<k\to\infty}\big(F^{0}_{k+1}/F^{0}_{k}\big)$, where $F^{0}_k\neq0$ for $k>k_0$, of the sequence $(F^{0}_k)_{k=k_0+1}^\infty$ exists and equals $\lambda_0$. Moreover,  part (i) of the theorem implies that for $k > k_0+n-1,$ we have $F^{ \bf a }_k\neq0$. Thus, it follows from formula \eqref{24}
that \begin{equation}
\begin{split}
\lim_{k_0<k\to\infty}\big(F^{\bf a}_{k+n}/F^{\bf a}_{k+n-1}\big)=\lim_{k_0<k\to\infty}\frac{a_0F^0_{k+n}+\sum_{i=1}^{n-1}a_{-i}\sum^{n-i}_{j=1}b_{i+j}F^0_{k+n-j}}{a_0F^0_{k+n-1}+\sum_{i=1}^{n-1}a_{-i}\sum^{n-i}_{j=1}b_{i+j}F^0_{k+n-1-j}}=\\
=\lim_{k_0<k\to\infty}\frac{a_0F^0_{k+n}/F^0_{k+n-1}+\sum_{i=1}^{n-1}a_{-i}\sum^{n-i}_{j=1}b_{i+j}F^0_{k+n-j}/F^0_{k+n-1}}{a_0+\sum_{i=1}^{n-1}a_{-i}\sum^{n-i}_{j=1}b_{i+j}F^0_{k+n-1-j}/F^0_{k+n-1}}=\\
=\lim_{k_0<k\to\infty}\frac{a_0\Phi^{0}+\sum_{i=1}^{n-1}a_{-i}\sum^{n-i}_{j=1}b_{i+j}(\Phi^{0})^{-j+1}}{a_0+\sum_{i=1}^{n-1}a_{-i}\sum^{n-i}_{j=1}b_{i+j}(\Phi^{0})^{-j}}=\Phi^{0}\!,
\label{26}
\end{split}
\end{equation}
i.e., the limit $\Phi^{\bf a}$ exists and equals $\lambda_0$. 

\end{proof}

\section{Conclusions}
\label{sec:III} 
A criterion proven in 1966 by Ostrowski \cite[Theorem 12.2]{Ostrowski} states that a linear recurrence with weights $b_i\ge0$, $i=1,\dots,n$, has an asymptotically simple polynomial with unique dominant zero $\lambda_0$ of multiplicity $\nu=1$ if the gcd of indices $j$ corresponding to positive weights  $b_j$ equals 1. Thus, in the case of a linear recurrence with nonnegative weights, our theorem allows establishing whether {\em all\,} sequences generated by such a linear recurrence from arbitrary nontrivial initial conditions have the same ratio limit $\lambda_0$ by simply finding the gcd of the positive weights.

In particular, our theorem implies that for a given $n$, all sequences with weights $(1,\dots,1)$ and arbitrary  initial conditions, e.g., the $n$-generalized Lucas sequence, have the same ratio limit as the $n$-generalized Fibonacci sequence. 
Also, it follows from our theorem that over 340 integer sequences with signatures $(m,\dots,m)$, $m\in\mathbb N$, and a variety of initial conditions, cataloged in the Sloane's {\em Online Encyclopedia of Integer Sequences} \cite{Sloane}, have their ratio limits equal to the dominant zero of the corresponding characteristic polynomial. Moreover, according to our theorem, the results proven by us for the weighted $n$-generalized Fibonacci sequences with weights $(p,\dots,p)$, $p>0$, and initial conditions $(0,0\dots,1)$ \cite{Szczyrba, Szczyrba1}, are, in fact, applicable for all sequences with such weights and arbitrary nontrivial initial conditions. 

The determination whether all sequences generated by a linear recurrences with complex weights have the same ratio limit equal to the unique zero $\lambda_0$ of its characteristic polynomial can be achieved by using a criterion introduced by Dubeau et al.~\cite{Dubeau1} that, as the authors showed, is valid in numerous cases, cf.~\cite[Theorem 15]{Dubeau1}. According to this criterion, if for any characteristic polynomial's zero $\lambda$ it holds 
\begin{equation}
\sum^{n-1}_{j=1}\Big|\sum^{n-1}_{i=j}\frac{b_{1+i}}{\lambda^{1+i}}\Big|<1, \label{31}
\end{equation}
 then the characteristic polynomial is asymptotically simple with the dominant zero $\lambda =\lambda_0$ of multiplicity $\nu=1$.

Of course, our theorem implies that {\em all integer\,} sequences that are generated by a linear recurrence that satisfies either of the criteria described above have the same ratio limit equal to the dominant zero $\lambda_0$ of the linear recurrence's characteristic polynomial. The OEIS \cite{Sloane} and Khovanova's website {\em Recursive Sequences} \cite{Khovanova} catalogs and describes applications of thousands of such integer sequences.

\bigskip \noindent MSC2010:  11B37, 11B39

\end{document}